\documentclass[11pt]{amsart} % Nicer than default article style: less
\usepackage{mathrsfs,cancel}          % \cancel{} command for crossing out a term
\usepackage{bm,amsmath,amsthm}     % Handy math stuff, theorem environments.
\usepackage{amssymb,empheq}            % Fancy math symbols.
\usepackage{euscript}           % Nice script font.
\usepackage{wasysym,eufrak}
\usepackage{graphicx,enumerate,calc,lscape,color}
\definecolor{red}{rgb}{1,0,0}
\usepackage[matrix,arrow,curve,frame]{xy}    % XY-pic diagram pac
\usepackage[margin=1in]{geometry}
\usepackage{hyperref}
\usepackage{bbm}
\usepackage{lineno}

\xymatrixcolsep{1.9pc}                          % Adjust size of diagrams.
\xymatrixrowsep{1.9pc}
\newdir{ >}{{}*!/-5pt/\dir{>}}                  % Make better tailed arrows

\newtheorem{thm}[subsection]{Theorem}
\newtheorem{defn}[subsection]{Definition}
\newtheorem{prop}[subsection]{Proposition}

\theoremstyle{definition}  % Bold headings and Roman body text.

  {\end{list}}

\newcommand{\F}{\mathbb{F}}

\numberwithin{equation}{subsection}

\DeclareFixedFont{\ttb}{T1}{txtt}{bx}{n}{12} % for bold
\DeclareFixedFont{\ttm}{T1}{txtt}{m}{n}{12}  % for normal

\usepackage{color}
\definecolor{deepblue}{rgb}{0,0,0.5}
\definecolor{deepred}{rgb}{0.6,0,0}
\definecolor{deepgreen}{rgb}{0,0.5,0}

\usepackage{listings}

\newcommand\pythonstyle{\lstset{
language=Python,
basicstyle=\ttm,
otherkeywords={self},             % Add keywords here
keywordstyle=\ttb\color{deepblue},
emph={MyClass,__init__},          % Custom highlighting
emphstyle=\ttb\color{deepred},    % Custom highlighting style
stringstyle=\color{deepgreen},
frame=tb,                         % Any extra options here
showstringspaces=false            % 
}}

\lstnewenvironment{python}[1][]
{
\pythonstyle
\lstset{#1}
}
{}

\newcommand\pythoninline[1]{{\pythonstyle\lstinline!#1!}}

\usepackage{tikz} %to draw nice pictures

\begin{document}
%\linenumbers

\title[a-number of hyperelliptic curves]{The a-number of hyperelliptic curves}
\author{Sarah Frei}

\maketitle

\begin{abstract}
It is known that for a smooth hyperelliptic curve to have a large $a$-number, the genus must be small relative to the characteristic of the field, $p>0$,  over which the curve is defined. It was proven by Elkin that for a genus $g$ hyperelliptic curve $C$ to have $a_C=g-1$, the genus is bounded by $g<\frac{3p}{2}$. In this paper, we show that this bound can be lowered to $g <p$. The method of proof is to force the Cartier-Manin matrix to have rank one and examine what restrictions that places on the affine equation defining the hyperelliptic curve. We then use this bound to summarize what is known about the existence of such curves when $p=3,5$ and $7$.
\end{abstract}

%%---------------{INTRODUCTION}----------------------------------------------------
\section{Introduction}
\label{intro}

Associated to an algebraic curve defined over a field of positive characteristic $p$ are a number of invariants used to better understand the structure of the curve, such as $p$-rank, Newton polygon, Ekedahl-Oort type, and $a$-number. Knowing if and when certain properties of a curve exist gives information about the moduli space of smooth projective curves of genus $g$ over a field $k$. Studied here is the $a$-number of hyperelliptic curves of genus $g$. The $a$-number $a_C$ of a hyperelliptic curve $C$ defined over an algebraically closed field $k$ of characteristic $p>0$ is $a_C=\text{dim}_k\text{Hom}(\alpha_p,\text{Jac}(C)[p])$, where $\alpha_p$ is the kernel of the Frobenius endomorphism on the additive group scheme $\mathbb{G}_a$. While the $a$-number of a curve is easily computible, there are still many open questions about this invariant.

For an algebraic curve of genus $g$ defined over $\mathbb{C}$, its Jacobian will have $p^{2g}$ $p$-torsion points. However, for a curve in characteristic $p$, the number of $p$-torsion points drops to $p^{f_C}$, where \linebreak$0 \le f_C \le g$. We define $f_C$ to be the $p$-rank of the curve. A generic curve of genus $g$ will have $f_C=g$. It must also be that the $a$-number is bounded above by $g-f_C$, so a typical curve of genus $g$ will have $a_C=0$. This means curves with larger $a$-numbers do not occur as often, and in fact curves with $a_C=g$ are very rare. An algebraic curve with $a_C=g$, called a superspecial curve, has the property that its Jacobian is isomorphic to a product of supersingular elliptic curves \cite{oort1975}. Because superspecial curves are as far from ordinary as possible, they are a popular topic for research.

For a curve to have a large $a$-number, the genus of that curve must be small relative to the characteristic $p>0$ of the field over which the curve is defined. It is a result of Ekedahl \cite{ekedahl1987} that for any curve with $a_C=g$, the genus is bounded by $g \le \displaystyle\frac{p(p-1)}{2}$. If the curve is hyperelliptic and $a_C=g$, then $g \le \displaystyle\frac{p-1}{2}$. 

If superspecial curves occur the least, then the next most infrequently occurring type of curve should be one with $a_C=g-1$. The next question that can be asked then is what kind of bound exists on the genus when $a_C=g-1$, and for any known bound, is that bound attained? It should be that the genus must still be small relative to the characteristic of the field. For a curve with $a_C=g-1$, it was shown by Re \cite{re2001} that $g \le p^2$. In fact, Re's results were more general, giving the bound  $g\le (g-a_C+1)\displaystyle\frac{p(p-1)}{2}+p(g-a_C)$ on the genus of a curve with any $a$-number. 

Further results by Elkin \cite{elkin2011} show that for a hyperelliptic curve with $a_C=g-1$, the bound on the genus is even lower: $g < \displaystyle\frac{3p}{2}$. Elkin's bound was also proven more generally, showing that if $g-a \le \displaystyle\frac{2g}{p}-2$, then there are no hyperelliptic curves of genus $g$ with $a_C \ge a$. Work by Johnston \cite{johnston2007} confirms Elkin's bound of $g < \displaystyle\frac{3p}{2}$.

While these general results are useful, it is not clear whether the bound is optimal for a given $a$-number. The goal of this paper is to explore this bound when $a_C=g-1$ and show that it can be lowered even further. The following result is proven in Section \ref{results}.

\begin{thm}
Let $g \ge p$ where $p$ is an odd prime. Then there are no smooth hyperelliptic curves of genus $g$ defined over a field of characteristic $p$ with $a$-number equal to $g-1$.
\end{thm}

These results show that for a hyperelliptic curve with $a=g-1$, the bound on the genus is even lower than was previously known. We must actually have $g < p$ for such a curve to exist. Section \ref{smallp} summarizes what this bound looks like for small fields. 

Based on computations for $p=5$, $p=7$ and $p=11$, it seems possible that this bound may be even lower when $p>3$. When $g=p-1$, for a genus $g$ hyperelliptic curve to have $a=g-1$ its affine equation $y^2=f(x)$ must take on a particular form. This is discussed in Section \ref{openquestion}. 

\subsection*{Acknowledgments}
I would like to thank my advisor Rachel Pries for her many helpful comments and suggestions on this paper, as well as for guiding me on this project while I was a graduate student at Colorado State University.

%%------{BACKGROUND}-------------------------------------------------------------------------------
\section{Background Information}

%\subsection{Hyperelliptic Curves and Their Jacobians}
%Let $k$ be an algebraically closed field of characteristic $p>0$. A hyperelliptic curve is a smooth curve $C$ which is a degree 2 cover of the projective line. It can be given by an affine equation $y^2=f(x)$ where $f(x)$ is a polynomial in $k[x]$. For $C$ to be smooth, $f(x)$ must be squarefree. The degree of $f(x)$ determines the genus of $C$, where a polynomial of degree $2g+1$ or $2g+2$ corresponds to a curve of genus $g$. Since the automorphism group of $\mathbb{P}^1$ acts triply transitively, which means any 3 points on $\mathbb{P}^1$ can be transformed to any other 3 points by an automorphism, we are allowed to pick up to 3 of the $2g+2$ branch points of $C$. Hence we will always fix a branch point at infinity and $f(x)$ will be of degree $2g+1$. Often, we will also fix $x=0$ as another branch point.
%
%The Jacobian of a hyperelliptic curve $C$ is a group Jac$(C)$ associated to the curve. It is defined as 
%$$\text{Jac}(C)=\displaystyle\frac{\text{Div}^0(C)}{\text{PDiv(C)}}$$
%where Div$^0(C)$ is the set of divisors on $C$ of degree 0, and PDiv$(C)$ is the set of principal divisors on $C$, that is, those that are linearly equivalent to a divisor of a meromorphic function on $C$.

%------------------------------------------
\subsection{The Cartier Operator}
Let $K=k(x,y)$ be the algebraic function field of a hyperelliptic curve $C$ given by $y^2=f(x)$, and let $d:K\rightarrow \Omega^1(K)$ be the canonical derivation of elements in $K$. For a holomorphic 1-form $\omega \in H^0(C,\Omega^1_C)$, we can write it as $\omega=d\phi + \eta^px^{p-1}dx$ with $\phi,\eta \in K$.

\begin{defn}
The modified Cartier operator $C':H^0(C,\Omega^1_C) \rightarrow H^0(C,\Omega^1_C)$ is defined for $\omega$ given as above by $C'(\omega)=\eta dx$.
\end{defn}

%The modified Cartier operator satisfies a number of basic properties:
%\begin{enumerate}
%\item $C'(\omega + \omega')=C'(\omega)+C'(\omega')$.
%\item $C'(\phi^p\omega)=\phi C'(\omega)$ for $\phi \in K$.
%\item $C'(\phi^{n-1})=d\phi$ if $n=p$, and 0 otherwise for $\phi \in K$.
%\item $C'(\omega)=0$ if and only if $\omega=d\phi$ for $\phi \in K$.
%\item $C'(\omega)=\omega$ if and only if $\omega=d\phi/\phi$ for $\phi \in K$.
%\end{enumerate}
%
%All of these properties can be proven directly from the definition, except for the last, which is shown in \cite{cartier1958}. 

\noindent For a full discussion on the Cartier operator as well as the modified Cartier operator, see \cite{yui1978}.

A canonical basis for $H^0(C,\Omega^1_C)$ is given by
$$\left\{\omega_i=\displaystyle\frac{x^{i-1}dx}{y} : 1\le i \le g \right\}.$$
We want to consider what the modified Cartier operator does to these basis elements. Recall that $C$ is given by $y^2=f(x)$, and if we let $f(x)^{(p-1)/2}=\displaystyle\sum_{j=0}^N \kappa_jx^j $ where $N=\displaystyle \frac{p-1}{2} (2g+1)$, then we can rewrite $\omega_i$ as follows:

\begin{align*}
\omega_i=& x^{i-1}y^{-p}y^{p-1}dx=y^{-p}x^{i-1}\displaystyle\sum_{j=0}^N \kappa_jx^j dx \\
 =& y^{-p}\left( \displaystyle\sum_{\substack{j \\ i+j \not\equiv 0 (\text{mod } p)}} \kappa_jx^{i+j-1}dx \right) + \displaystyle\sum_{l} \kappa_{(l+1)p-i} \displaystyle\frac{x^{lp}}{y^p}x^{p-1}dx.
\end{align*}

The highest possible power of $x$ is $N+i-1$, so $lp+p-1 \le N+i-1$, which forces 
$$0 \le l \le \displaystyle\frac{N+i}{p}-1=g-\displaystyle\frac{1}{2}-\left(\displaystyle\frac{2g-2i+1}{2p}\right) < g-\displaystyle\frac{1}{2}.$$

This means the sum in the second term is over $0\le l \le g-1$. Thus we can now see that 
$$C'(\omega_i)=\displaystyle\sum_{l=0}^{g-1} \kappa_{(l+1)p-i}^{1/p} \displaystyle\frac{x^l}{y}dx.$$

This shows that $C'$ is a map on $H^0(C,\Omega^1_C)$ and we can represent its action on the basis with a matrix. If we write $\bar{\omega}=(\omega_1,...,\omega_g)$, then 
$$C'(\bar{\omega})=A^{(1/p)}\bar{\omega}$$
where $A$ is a $g\times g$ matrix $[a_{ij}]$ with $a_{ij}=\kappa_{pi-j}$.

\begin{defn}
The matrix $A$ described above is the Cartier-Manin matrix of the hyperelliptic curve $C$ of genus $g$ defined over $k$.
\end{defn}

%----------------------------------------------
\subsection{P-Rank and A-Number}
%We first define an $A$-group scheme $G$ to be a group object with the group structure described by homomorphisms on a locally free algebra $A$ over a commutative ring $R$. The group structure is given by the maps $\mu:A\rightarrow A \otimes_R A$, $\varepsilon: A \rightarrow R$ and $i:A\rightarrow A$ which define the multiplication, identity, and inverse laws, respectively. For the purposes of this paper, $R$ will be an algebraically closed field $k$. The group scheme $\mu_p\cong \text{Spec}(k[x]/(x-1)^p)$ is the kernel of the Frobenius endomorphism on the multiplicative group $\mathbb{G}_m=\text{Spec}(k[x,x^{-1}])$. The group scheme $\alpha_p\cong \text{Spec}(k[x]/x^p)$ is the kernel of the Frobenius endomorphism on the additive group $\mathbb{G}_a=\text{Spec}(k[x])$. For more on group schemes, see \cite{tate1997}.

The group scheme $\mu_p\cong \text{Spec}(k[x]/(x-1)^p)$ is the kernel of the Frobenius endomorphism on the multiplicative group $\mathbb{G}_m=\text{Spec}(k[x,x^{-1}])$. The group scheme $\alpha_p\cong \text{Spec}(k[x]/x^p)$ is the kernel of the Frobenius endomorphism on the additive group \linebreak$\mathbb{G}_a=\text{Spec}(k[x])$. For more on group schemes, see \cite{tate1997}.

The $p$-rank of a hyperelliptic curve $C$ is $f_C=\text{dim}_k\text{Hom}(\mu_p,\text{Jac}(C)[p])$.  An equivalent definition of the $p$-rank is that it is the positive integer $f_C$ such that Jac$(C)[p](k)\cong (\mathbb{Z}/p\mathbb{Z})^{f_C}$, so $\#\text{Jac}(C)[p](k)=p^{f_C}$. We see that $0 \le f_C \le g=\text{dim}(\text{Jac}(C))$. A curve is called ordinary if $f_C=g$, and non-ordinary otherwise.

The $a$-number of $C$ is $a_C=\text{dim}_k\text{Hom}(\alpha_p,\text{Jac}(C)[p])$. We also have $0\le a_C \le g$, and in fact $a_C \le g-f_C$. Curves with $a_C=g$ are called superspecial and do not occur often, due to the fact that a typical curve of genus $g$ has $f_C=g$. Curves with $a_C=g-1$ are forced to have $f_C=0$ or $f_C=1$ which limits their occurrences.

The $a$-number is also related to the rank of the Cartier-Manin matrix introduced above. For an abelian variety $X$ of dimension $g$, such as the Jacobian of a genus $g$ hyperelliptic curve, the Frobenius operator $F:X\rightarrow X^{(p)}$ is the $p$-th power map on $X$, and the Verschiebung operator $V:X^{(p)} \rightarrow X$ is the map such that $V \circ F = [p]$, the multiplication-by-$p$ map. The $a$-number is also defined \cite{li1998} as the dimension of the kernel of the action of $V$ on $H^0(X,\Omega^1_X)$. If we let $v=\text{dim}VH^0(X,\Omega^1_X)$, this gives us that $a_C=g-v$. It is also known for a smooth projective curve $C$, such as a hyperelliptic curve, that the action of the Cartier operator on $H^0(C,\Omega^1_C)$ agrees with the action of $V$ on $H^0(\text{Jac}(C),\Omega^1_{\text{Jac}(C)})\cong H^0(C,\Omega^1_C)$ \cite{oda1969}. Since we can express the action of the Cartier operator on $H^0(C,\Omega^1_C)$ with the Cartier-Manin matrix $A$, we see that $a_C=g-\text{rank}(A)$. 

It turns out that associated with any abelian variety $X$ of dimension $g$ is a short exact sequence
$$0 \rightarrow H^0(X,\Omega^1_{X})\rightarrow H^1_{dR}(X) \rightarrow H^0(X,\Omega^1_{X})\rightarrow 0.$$
The Frobenius operator acts on $H^0(X,\Omega^1_X)$ in this sequence, and the Verschiebung operator acts on $H^1_{dR}(X)$ so $H^0(X,\Omega^1_X)=VH^1_{dR}(X)$. 

For the sake of notation, we will let $a_C=a$ for the rest of this paper. In studying hyperelliptic curves with $a=g-1$, we will thus be looking for curves with a Cartier-Manin matrix of rank one. We will utilize the fact that for a matrix of rank 1, there is at least one non-zero entry, and every $2\times 2$ minor has determinant 0. %This ensures that all of the rows, or equivalently all of the columns, are linearly dependent. 

%%------{RESULTS}-------------------------------------------------------------------------------
\section{Results} \label{results}

In this section we will use the following notation. Let $C$ be a hyperelliptic curve given by the equation $y^2=f(x)$ where $f(x)=\displaystyle\sum_{i=1}^{2g+1}c_ix^i$ with $c_i \in k$ where $k$ is an algebraically closed field of characteristic $p>0$. Note that by a change of variables, we can assume $c_0=0$ and $c_{2g+1}=1$. We will assume that $C$ has $a=g-1$. Then we will define the coefficients $\kappa_i$ as follows:
$$f(x)^{(p-1)/2} = \displaystyle\sum_{i=0}^{\left(\frac{p-1}{2}\right)(2g+1)}\kappa_ix^i$$
and $\kappa_i=0$ if $i<\frac{p-1}{2}$ or $i>(\frac{p-1}{2})(2g+1)$. The Cartier-Manin matrix $A$ associated to $C$ is the $g \times g$ matrix $[a_{ij}]$ where $a_{ij}=\kappa_{pi-j}$. We will denote row $m$ of $A$ by $A_m$. For $C$ to have $a$-number equal to $g-1$, $A$ must have rank one.

\begin{thm}
\label{mainthm}
Let $g \geq p$ where $p$ is an odd prime. Then there are no smooth hyperelliptic curves of genus $g$ defined over an algebraically closed field of characteristic $p$ with $a$-number equal to $g-1$.
\end{thm}

\begin{proof}
We will proceed by considering two separate cases: first when $g>p$ and then when $g=p$.

{\bf Case 1:} Let $g>p$ where $p$ is an odd prime. We consider the entries $a_{i,j}=\kappa_{pi-j}$ of the Cartier-Manin matrix $A$.
Since $\kappa_i=0$ for $0 \le i \le \frac{p-3}{2}$, $a_{1,j}$ is possibly nonzero for $1 \le j \le \frac{p+1}{2}$, and $a_{1,j}=0$ for $\frac{p+3}{2} \le j \le g$. 
The largest nonzero term of $f(x)^{(p-1)/2}$ is $x^{g(p-1)+(p-1)/2}$, so $\kappa_{g(p-1)+(p-1)/2}=\kappa_{gp-(g-(p-1)/2)}=1$ and any larger-indexed coefficient is zero. This means $a_{g,j}=0$ for $1 \le j \le g-\frac{p+1}{2}$, and $a_{g,j}$ is possibly nonzero for $g-\frac{p-1}{2} \le j \le g$.

Now let us suppose that $g=p+m$ for some integer $m \ge 1$. We have 
$$a_{1,(p+1)/2}=\kappa_{(p-1)/2}=c_1^{(p-1)/2},$$
and $a_{1,(p+1)/2+m}=0$, since $a_{1,(p+1)/2}$ is the last nonzero entry in $A_1$. Also, $a_{g,(p+1)/2}=0$, since $a_{g,j}=0$ for $1 \le j \le g-\frac{p+1}{2}=\frac{p-1}{2}+m$ and $m \ge 1$. Hence $a_{g,(p+1)/2}$ is possibly the last zero term in $A_g$, if $m=1$. Lastly, $a_{g,(p+1)/2+m}=1$, since $g-\frac{p-1}{2}=p+m+\frac{p-1}{2}=\frac{p+1}{2}+m$, which is the first non-zero term in $A_g$. Using this $2\times 2$ minor, we get $$a_{1,(p+1)/2} \cdot a_{g,(p+1)/2+m} - a_{g,(p+1)/2} \cdot a_{1,(p+1)/2+m} =0,$$ which means $c_1^{(p-1)/2}\cdot 1 - 0 \cdot 0=0$. This forces $c_1=0$.  But then $f(x)=\displaystyle\sum_{i=2}^{2g+1} c_ix^i=x^2\displaystyle\sum_{i=2}^{2g+1}c_ix^{i-2}$ is not squarefree and $C$ is not a smooth curve. Therefore, when $g>p$ there are no smooth hyperelliptic curves of genus $g$ defined over a field of characteristic $p$ with $a$-number equal to $g-1$.

{\bf Case 2: } Let $g=p$ where $p$ is an odd prime. We again consider the $a_{i,j}$ in the Cartier-Manin matrix. There will be $g-\frac{p+1}{2}$ zeros in $A_1$ and $A_g$. For $g=p$, this means the last $\frac{p-1}{2}$ entries of $A_1$ are zeros and the first $\frac{p-1}{2}$ entries of $A_g$ are zeros. As above, $\kappa_{\frac{p-1}{2}}=c_1^{(p-1)/2}$ and $\kappa_{(2g+1)(p-1)/2}=\kappa_{(2p^2-p-1)/2} =1$. We will assume $c_1 \neq 0$ so that $C$ is not singular at $x=0$. This gives us an idea of what $A$ looks like:

\[ \left( \begin{array}{ccccccc}
\kappa_{\frac{p-1}{2}+\frac{p-1}{2}} & \ldots & \kappa_{\frac{p-1}{2}+1} & c_1^{(p-1)/2} & 0 & \ldots & 0 \\
 & \ldots &  & \kappa_{\frac{p-1}{2}+p} & \kappa_{\frac{p-1}{2}+(p-1)} & \ldots & \kappa_{\frac{p-1}{2}+\frac{p+1}{2}} \\
\vdots & \ddots & \vdots & \vdots & \vdots & \ddots  & \vdots \\
\kappa_{\frac{2p^2-p-1}{2}-\frac{p+1}{2}} & \ldots & \kappa_{\frac{2p^2-p-1}{2}-(p-1)} & \kappa_{\frac{2p^2-p-1}{2}-p} & &  \ldots &  \\
0 & \ldots & 0 & 1 & \kappa_{\frac{2p^2-p-1}{2} -1} & \ldots & \kappa_{\frac{2p^2-p-1}{2}-\frac{p-1}{2}} \\
\end{array} \right) \]

We can again consider the $2\times 2$ minors of $A$, or we can simply use the fact that because rk$A=1$, every column of $A$ is a scalar multiple of the middle column. The columns to the left of the middle column must be zero since the last entry of the index $\frac{p-1}{2}$ column is 1 while the last entry of the previous columns is zero. The columns to the right of the middle column must also be zero since the first entry of the index $\frac{p-1}{2}$ column is $c_1^{(p-1)/2}\neq 0$ while the first entry of the following columns is zero. This means $f(x)^{(p-1)/2}$ has the following form:
$$f(x)^{(p-1)/2} = \sum_{i=0}^{p-1} \kappa_{\frac{p-1}{2}+ip}x^{(p-1)/2+ip}=x^{(p-1)/2}h(x^p)=x^{(p-1)/2}\widetilde{h}(x)^p,$$
where $h(x)=\sum_{i=0}^{p-1}\kappa_{\frac{p-1}{2}+ip}x^{i}$ and where the last equality is a consequence of the multinomial theorem in characteristic $p>0$. Thus, since $f(x)=x(c_1+c_2x+...+x^{2p})=x\widetilde{f}(x)$, we see that $\widetilde{f}(x)^{(p-1)/2}=\widetilde{h}(x)^p$. Then we see that any root of $\widetilde{h}$ is a root of $\widetilde{f}^{(p-1)/2}$ with multiplicity $p$, making it a root of $\widetilde{f}$ with multiplicity greater than 1. Thus $f$ is not squarefree and hence $C$ is a singular hyperelliptic curve. Therefore, when $g=p$ there are no smooth hyperelliptic curves of genus $g$ defined over a field of characteristic $p$ with $a$-number equal to $g-1$.
\end{proof}

%%------{COMPUTATIONS}-------------------------------------------------------------------------------
\section{Computations and Examples for Small Primes}
\label{smallp}

\subsection{For $p=3$} 
\label{peq3}
We see from Elkin's bound that hyperelliptic curves defined over $\overline{\mathbb{F}}_3$ with $a=g-1$ will only occur when $g< 5$. By Theorem \ref{mainthm}, in fact such a curve will only occur for $g<3$. Genus $3$ hyperelliptic curves have been studied extensively, and it was previously known that curves with $a=2$ do not exist \cite{elkin2007}. It is also known that genus 2 hyperelliptic curves with $a=1$ exist for all $p \geq 3$. Hence for $p=3$, genus 2 hyperelliptic curves are the only hyperelliptic curves with $a=g-1$.

%------------------------------
\subsection{For $p=5$}
\label{peq5}
According to Elkin's bound, hyperelliptic curves with $a=g-1$ will only occur when $g < \frac{15}{2}$. For $p=5$ it is known that such hyperelliptic curves exist with genus 2 and with genus 3 \cite{elkin2007}. When $g=3$, they in fact occur with both $p$-rank 0 and 1.

It is next worth investigating $g=4, 5, 6$, and $7$, but Theorem \ref{mainthm} in Section \ref{results} shows that for $g=5,6$ and $7$, there are no smooth hyperelliptic curves of such a genus with $a=g-1$. It can be shown that if we assume $C$ is a genus 4 hyperelliptic curve with $a=3$ defined by $y^2=f(x)$, then $f(x)=x(x+2c_8)^3(x+\sqrt[5]{c_4})^5$. This means there are no smooth hyperelliptic curves of $g=4$ with $a=3$ defined over a field of characteristic 5. Hence, the case $p=5$ is completely determined, with curves having $a=g-1$ only existing when $g=2$ and $g=3$.

%-------------------------------------------
\subsection{For $p=7$}
\label{peq7}
Elkin's bound for $p=7$ gives that for a hyperelliptic curve with $a=g-1$, we must have $g<\frac{21}{2}$, so we are interested in looking for curves with genus up to 10. Theorem \ref{mainthm} shows that such a curve will not exist with $g\ge p$, so in fact we only need to study $g=2,3,4,5$ and 6. It was previously shown that genus 2 curves exist with $a=1$ in characteristic 7.

Hyperelliptic curves of genus 3 with $a=2$ exist, and as occurs for $p=5$, they exist with $p$-rank both 0 and 1. In this case, as expected, there are far more such curves with $p$-rank 1 than $p$-rank 0 defined over $\mathbb{F}_7$.

It is still unknown whether or not curves of genus $4$ exist with $a=3$. The Sage code shown in Section \ref{Sagecode} was used to determine that genus 4 hyperelliptic curves with a defining polynomial of the form $f(x)=c_1x+c_2x^2+...+c_8x^8+x^9$ do not exist over $\mathbb{F}_7$. We note that there could still exist a curve with either $c_0\neq 0$, $c_9\neq 1$ or $c_{10}\neq 0$ and the desired $a=3$ defined over $\mathbb{F}_7$, so this was not an exhaustive search. After checking $1,000,000$ hyperelliptic curves defined over $\mathbb{F}_{49}$ with branch points fixed at $x=0, 1$ and $\infty$, none were found to have $a=3$. This code can also be seen in Section \ref{Sagecode}. However, this is a very small portion of the total number of curves defined over $\mathbb{F}_{49}$, and it is possible that such a curve could exist over a larger extension of $\mathbb{F}_7$.
%However, this is a very small portion of the total number of curves defined over $\mathbb{F}_{49}$, so it is possible that such a curve does still exist. Furthermore, not finding any over $\mathbb{F}_{49}$ does not mean such a curve doesn't still exist over a larger extension, although it does mean that the occurrence is not very likely.

%\begin{figure}
%\label{fig2}
%\caption{Computations in Sage show there are no genus 4 hyperelliptic curves with a defining polynomial of the given form with $a=3$ over $\mathbb{F}_7$ and that a random check of 1,000,000 curves over $\mathbb{F}_{49}$ did not find any genus 4 curves with $a=3$.}
%\includegraphics[scale=0.75]{curvesp7g4_1.png}
%\end{figure}

When $g=5$, we see similar results. It is still open whether or not curves of genus 5 exist with $a=4$. It has been checked in Sage that there are no such hyperelliptic curves over $\mathbb{F}_7$ with defining polynomial of the form $f(x)=c_1x+...+c_{10}x^{10}+x^{11}$ (again, a non-exhaustive search). We next checked for curves branched at 0 and $\infty$ defined over $\mathbb{F}_{49}$. In this case, we use information from the Cartier-Manin matrix, again forcing the matrix to have rank one, to further shrink the search space. After checking 30,000,000 random curves under these restrictions, none were found to have $a=4$. We note again that this is only a small portion of the curves defined over $\F_{49}$, and those checked were only curves in a restricted search space, since there could exist a genus 5 hyperelliptic curve defined over $\mathbb{F}_{49}$ with no rational ramification points having $a=4$. 

For genus 6 curves, it can be shown that if we assume $C$ is a genus 6 hyperelliptic curve with $a=5$ defined by $y^2=f(x)$, then $f(x)= x(x+3c_{12})^5(x+\sqrt[7]{c_6})^7$. Thus, there are no smooth hyperelliptic curves of genus 6 with $a=5$ when $p=7$.

%%------{OPEN QUESTION}-------------------------------------------------------------------------------
\section{Further Lowering the Bound}
\label{openquestion}

Without any known examples of algebraic curves of genus $g>3$ with $a=g-1$, it is unclear whether or not it is possible to lower the bound on the genus any further. Future work in this area could include exploring the cases of $g=p-1$ and $g=p-2$. 

As stated in Sections \ref{peq5} and \ref{peq7}, neither smooth hyperelliptic curves of genus 4 with $a=3$ nor smooth hyperelliptic curves of genus 6 with $a=5$ exist when $p=5$ or $p=7$, respectively. It can also be shown that if we assume $C$ is a genus 10 hyperelliptic curve with $a=9$ defined by $y^2=f(x)$ in characteristic 11, then $f(x)=x(x+5c_{20})^9(x+\sqrt[11]{c_{10}})^{11}$, and hence $C$ is not smooth. These cases suggest that curves with $a=g-1$ likely do not exist when $g=p-1$. In fact, we have the following result.

\begin{prop}
\label{thmp-1}
Let $C$ be a hyperelliptic curve defined over a field of characteristic $p>3$ of genus $g=p-1$, where $C$ is defined above. If $C$ has $a=g-1$, then $f(x) \in k[x,c_g,c_{2g-g/2},c_{2g}]$.
\end{prop}

Thus, for a hyperelliptic curve $C$ with $a=g-1$ to exist when $g=p-1$, its affine equation \linebreak$y^2=f(x)$ must take on a very specific form; the polynomial $f(x)$ is completely determined by only three of its $2g$ coefficients. Proposition \ref{thmp-1} is proven using the same methods employed in Section \ref{results}, where the associated Cartier-Manin matrix is assumed to have rank 1, and the relationships forced on the coefficients of $f(x)$ are studied. 

As shown in Section \ref{peq7}, it seems possible that curves of genus 5 with $a=4$ do not exist in characteristic 7. It would be worth generating data for $p=11$ and $g=9$ to explore the existence of hyperelliptic curves with $a=8$. From there, an attempt could be made to make a general statement about the existence of hyperelliptic curves of genus $g=p-2$ and $a=g-1$ when $p>5$.

%%%------{FUTURE WORK}-------------------------------------------------------------------------------
\section{Sage Code}
\label{Sagecode}

The following is an sample of some of the code used to obtain results discussed in Section \ref{peq7}. In both examples listed here, the returned output was $N=0$. \\

\begin{python}
F=GF(7)
R.<x>=PolynomialRing(F)
N=0
V=VectorSpace(F, 8)
for m in V:
    f=m[1]*x+m[2]*x^2+m[3]*x^3+m[4]*x^4+m[5]*x^5+m[6]*x^6+m[7]*\
    	x^7+m[0]*x^8+x^9
    if f.is_squarefree()==True:
    	C=HyperellipticCurve(f)
    	B=C.Cartier_matrix()
    	if B.rank()==1:
    	    N=N+1;
    	    C
N
\end{python}

\begin{python}
F=GF(49, 'a')
R.<x>=PolynomialRing(F)
N=0
for i in range(1000000):
    m=random_vector(F, 7)
    f=(x-1)*(m[0]*x+m[1]*x^2+m[2]*x^3+m[3]*x^4+m[4]*x^5+m[5]*x^6\
    	+m[6]*x^7+x^8
    if f.is_squarefree()==True:
    	C=HyperellipticCurve(f)
    	B=C.Cartier_matrix()
    	if B.rank()==1:
    	    N=N+1;
    	    C
N
\end{python}

%%------{BIBLIOGRAPHY}-----------------------------------------------------------------------------------
\bibliographystyle{alpha}
\bibliography{article_final}

{\small {\bf \it E-mail: } sfrei@uoregon.edu }

\end{document}